\pdfoutput=1

\documentclass[reqno]{amsart}

\usepackage[latin1]{inputenc}

\usepackage{amssymb}
\usepackage{amsmath}
\usepackage{amscd}
\usepackage{amsthm}
\usepackage{amsfonts}
\usepackage{enumerate}
\usepackage{graphicx}
\usepackage{url}
\usepackage[breaklinks=true,hyperref]{hyperref}
\usepackage{amssymb}
\usepackage[dvips]{color}
\usepackage{epsfig}
\usepackage{mathrsfs}

\newtheorem{theorem}{Theorem}[section]
\newtheorem{lemma}[theorem]{Lemma}

\newtheorem{corollary}[theorem]{Corollary}
\newtheorem{conjecture}[theorem]{Conjecture}

\theoremstyle{definition}

\newtheorem{definition}[theorem]{Definition}

\theoremstyle{remark}

\newtheorem{remark}[theorem]{Remark}
\numberwithin{equation}{section}

\date{}
\author{Jon Schneider}
\title{Enumeration and quasipolynomiality of chip-firing configurations}
\begin{document}
\maketitle

\begin{abstract}
In this paper we explore enumeration problems related to the number of reachable configurations in a chip-firing game on a finite connected graph $G$. We define an auxiliary notion of debt-reachability and prove that the number of debt-reachable configurations from an initial configuration with $c$ chips on one vertex is a quasipolynomial in $c$. For the cycle graph $C_n$, we apply these results to compute a near explicit formula for the number of debt-reachable configurations. We then derive polynomial asymptotic bounds for the number of debt-reachable and reachable configurations, and finally provide evidence for a quasipolynomiality conjecture regarding the number of reachable configurations.
\end{abstract}

\section{Introduction}

On a finite connected graph $G$, we place a nonnegative integer amount of chips at each of the vertices. A vertex can fire if it contains at least as many chips as its degree, in which case it sends one chip to each adjacent vertex. Such a move is called a chip-firing move. Chip-firing on graphs is known to have applications to numerous related areas of mathematics, such as tropical geometry \cite{Ha}, statistical physics \cite{Ma}, and combinatorics \cite{Ho}.

In this paper, we consider the notions of reachability and debt-reachability and related enumerative problems. A configuration of chips $C'$ is reachable from another configuration $C$ if there is a sequence of chip-firing moves that transforms configuration $C$ to configuration $C'$. If we ignore the restriction that a vertex must contain at least as many chips as its degree, we obtain the related concept of debt-reachability (see Section \ref{definitions} for a precise definition). 

For a fixed graph $G$ with $n$ vertices, in Theorem \ref{thm4} we show that the number of configurations debt-reachable from a configuration with $c$ chips on one vertex is a quasipolynomial in $c$ with degree $n-1$ and period dividing $\kappa(G)$. In Theorems \ref{thm7_5} and \ref{thm12}, we show that both the numbers of configurations debt-reachable and reachable from a configuration with $c$ chips grow asymptotically as 

$$\dfrac{c^{n-1}}{(n-1)!\kappa(G)}$$

Finally, based on empirical evidence we conjecture that the number of configurations reachable from a configuration with $c$ chips on one vertex is a quasipolynomial in $c$ for sufficiently large values of $c$ (Conjecture \ref{conj1}). 

In addition to this, we also prove various properties about the number of such configurations for the case where $G$ is a cycle graph $C_n$. For $n=3$, in Theorems \ref{thm1} and \ref{thm3} we compute an explicit expression for both the number of reachable configurations and debt-reachable configurations. For $n > 3$, in Theorem \ref{thm10_5} we prove that the number of configurations debt-reachable from the configuration with $c$ chips on one vertex is equal to

$$\dfrac{1}{n}\binom{c+n-1}{n-1} + f(c \bmod n)$$

where $f(c \bmod n)$ is a constant which depends only on the remainder of $c$ modulo $n$. 

The rest of this paper is organized as follows. In Section~\ref{definitions} we provide background and definitions. In Section \ref{sec3}, we examine chip-firing on $C_3$ and prove Theorems \ref{thm1} and \ref{thm3} mentioned above. In Section \ref{sec4}, we prove quasipolynomiality and derive an asymptotic bound to the number of debt-reachable configurations. In Section \ref{sec5}, we apply the results of Section \ref{sec4} to the class of cycle graphs $C_n$. In Section \ref{sec6}, we prove that the number of reachable configurations is asymptotically equal to the number of debt-reachable configurations. Finally, in Section \ref{sec7}, we state Conjecture \ref{conj1} and provide some progress towards a proof.

\section{Background and Definitions}\label{definitions}

Unless otherwise stated, when we refer to a graph $G$, $G$ is finite, undirected, and connected with $n$ vertices. We discount the possibility of loops or multiedges; $G$ will always be a simple graph. We will also implicitly assume the prior existence of some labelling of the vertices of $G$ from $1$ through $n$. In accordance with common notation, we will let $C_n$ refer to the cycle graph on $n$ vertices, $K_n$ refer to the complete graph on $n$ vertices, and $P_n$ refer to the path graph on $n$ vertices.

The \textit{degree} of a vertex in a graph is equal to the number of other vertices it shares an edge with. We will denote the degree of vertex $i$ by $d_i$. 

\begin{definition}
The \textit{adjacency matrix} $A$ (alternatively $A(G)$) of a graph $G$ is the $n$ by $n$ matrix which has a $1$ in the $i$th row and $j$th column if there is an edge connecting vertices $i$ and $j$ and a $0$ otherwise.
\end{definition}

\begin{definition}
The \textit{degree matrix} $D$ (alternatively $D(G)$) of a graph $G$ is the $n$ by $n$ diagonal matrix whose $i$th element along the main diagonal is equal to $d_i$.
\end{definition}

\begin{definition}
The \textit{Laplacian matrix} $L$ (alternatively $L(G)$) of a graph $G$ is the $n$ by $n$ matrix equal to $D - A$.
\end{definition}

It is easily shown that, if $G$ is connected, $L$ has rank $n-1$ with a zero eigenvalue of $(1,1,\dots,1)^t$. We will therefore often let $L'$ denote any of the $n-1$ by $n-1$ matrices formed by removing one row and column of $L$ (usually, we will assume without loss of generality that this is the last row and column). We call $L'$ the \textit{reduced Laplacian matrix} of $G$.

\begin{definition}
A \textit{spanning tree} of $G$ is an acyclic connected subgraph of $G$ with $n$ vertices.
\end{definition}

We will let $\kappa(G)$ count the number of spanning trees of $G$. A famous theorem of Kirchoff states that \cite{St}:

\begin{theorem}[Matrix-Tree Theorem]\label{mtthm}
For a graph $G$, 

$$\kappa(G) = |\det L' |$$
\end{theorem}

We proceed to define the notion of chip firing on a graph.

\begin{definition}
For a graph $G$, we call an $n$-tuple of nonnegative integers $(c_1, c_2, c_3, ..., c_{n})$ a \textit{chip-firing configuration} on $G$ with $c_i$ \textit{chips} on vertex $i$. If we remove the nonnegativity constraint and allow the $c_i$ to be arbitrary integers, we call such an $n$-tuple a \textit{general chip-firing configuration}. If at least one of the $c_i < 0$, then we say the configuration is \textit{in debt}, or is a \textit{debt chip-firing configuration}.
\end{definition}

\begin{definition}
We say that vertex $i$ \textit{can fire} if $c_i \geq d_i$. In this case, vertex $i$ loses $d_i$ chips and every vertex adjacent to vertex $i$ gains $1$ chip. We call this operation a \textit{legal chip-firing move}. If we remove the condition that we must have $c_i \geq d_i$, then we call the resulting operation a \textit{general chip-firing move}, or just a \textit{chip-firing move}. For the specific case where $c_i < d_i$, we will call the operation a \textit{debt chip-firing move}. A debt chip-firing move always results in a debt configuration.
\end{definition}

Note that the total number of chips on a graph stays constant under a chip-firing move. Moreover, as an operator on the space of chip-firing configurations of a graph $G$, firing vertex $i$ acts by translating the current configuration by the $i$th row of $L(G)$. That is, firing vertex $i$ sends configuration $C$ to the configuration $C + Le_i$, where $e_i$ is the $i$th unit vector. A corollary of this is that chip-firing moves are abelian; firing vertex $i$ and then firing vertex $j$ will result in the same configuration as firing vertex $j$ and then firing vertex $i$. Similarly, we also have the following consequence:

\begin{theorem}\label{triv3}
Firing each of the vertices once leaves the configuration invariant.
\end{theorem}
\begin{proof}
This follows from the fact that $(1,1,\dots,1)^t$ is a zero eigenvalue of the Laplacian, as noted above.
\end{proof}

\begin{definition}\label{defreachable}
We say that a chip configuration $C'$ is \textit{reachable} from a chip configuration $C$ if there exist a sequence of legal chip-firing moves which transform $C$ into $C'$. We say that a chip configuration $C'$ is \textit{debt-reachable} from a chip configuration $C$ if there exists a sequence of chip-firing moves (possibly debt chip-firing moves) which transforms $C$ into $C'$. Note that both $C$ and $C'$ must be non-debt chip configurations.
\end{definition}

As a direct consequence of Definition \ref{defreachable} and the preceding observation, we have the following useful check for debt-reachability:

\begin{theorem}\label{triv2}
A configuration $C$ is debt-reachable from the configuration $C'$ iff there exists a solution $x \in \mathbb{Z}^n$ to $Lx = C - C'$. 
\end{theorem}

Theorem \ref{triv2} allows us to define the following:

\begin{definition}
A vector $x \in \mathbb{Z}^n$ is a \textit{firing vector} for the transition from $C'$ to $C$ if $Lx = C - C'$. If $x = (x_1, x_2, \dots, x_n)$, then $x_i$ corresponds to vertex $i$ being fired $x_i$ times. Since the only zero eigenvector of $L$ is $(1,1,\dots,1)^t$, if $Lx = C - C'$ has a solution, there is a unique solution such that $\min(x_i) = 0$. We call this unique solution the \textit{reduced firing vector} for the transition from $C'$ to $C$.
\end{definition}

Finally, we define some terminology related to the theory of Ehrhart polynomials. Let an integral polytope be a polytope whose vertices have integer coordinates.

\begin{definition}\label{ehrpoly}
The \textit{Ehrhart polynomial} of an integral polytope $P$ in $\mathbb{R}^n$ is defined to be:

$$i(P, q) = \# (qP \cap \mathbb{Z}^n)$$

That is, $i(P,q)$ is equal to the number of lattice points inside the $q$-fold expansion of the polytope $P$. It is a theorem of Ehrhart \cite{Be} that $i(P,q)$ is indeed a polynomial of degree $n$ in $q$ for integer values of $q$.
\end{definition}

Likewise, let a rational polytope be a polytope whose vertices have rational coordinates. 

\begin{definition}
A \textit{quasipolynomial} $q(k)$ is a function from $\mathbb{Z}$ to $\mathbb{Z}$ of the form:

$$q(k) = c_d(k)k^d + c_{d-1}(k)k^{d-1}+\dots+c_0(k)$$

where the $c_i(k)$ are periodic functions of $k$ with integer period. We call the period of $q$ the lowest common multiple of the periods of the $c_i(k)$. 
\end{definition}

Then we also have the following definition.

\begin{definition}
The \textit{Ehrhart quasipolynomial} of an rational polytope $P$ in $\mathbb{R}^n$ is defined to be:

$$i(P, q) = \# (qP \cap \mathbb{Z}^n)$$

As in Definition \ref{ehrpoly} it is a theorem of Ehrhart \cite{St} that $i(P,q)$ is indeed a quasipolynomial of degree $n$ in $q$ for integer values of $q$. Moreover, if $kP$ is an integer polytope, then the period of $i(P,q)$ divides $k$.
\end{definition}

\section{Enumeration of reachable configurations on $C_3$}\label{sec3}

In this section, we explicitly compute the number of debt-reachable and reachable configurations from the configuration $(c,0,0)$ on the triangle graph $C_3$. The techniques we use to do this are also illustrative of the techniques we will apply later in the paper to more general graphs.

\begin{theorem}\label{thm1}
The number of configurations that are debt-reachable from $(c,0,0)$ is equal to:

$$\dfrac{3k^2 + 3k + 2}{2}$$

if $c = 3k$,

$$\dfrac{3k^2 + 5k + 2}{2}$$

if $c = 3k+1$, and

$$\dfrac{3k^2 + 7k + 4}{2}$$

if $c = 3k+2$

\end{theorem}

\begin{proof}
We wish to compute how many configurations are debt-reachable from $(c,0,0)$. Since we are only concerned with debt-reachability, we do not care about the order of the firings, so we can just assume that the firing vector for this transition is equal to $(x_1, x_2, x_3)$, for $x_1, x_2, x_3 \in \mathbb{Z}$. After such an operation, we are in the configuration given by:

$$(c-2x_1 + x_2 + x_3, x_1 - 2x_2 + x_3, x_2 - 2x_3 + x_1)$$

In order for this configuration to be valid, all components must be greater than or equal to $0$. This gives us the following set of inequalities.

\begin{eqnarray*}
-2x_1 + x_2 + x_3 + c &\geq & 0 \\
x_1 -2x_2 + x_3 &\geq & 0\\
x_2 - 2x_3 + x_1 &\geq & 0
\end{eqnarray*}

Now, since firing vectors are equivalent modulo translation by $(1,1,1)^t$ (from Theorem \ref{triv3}), set $x_1 = 0$. This gives us the following set of inequalities in two variables:

\begin{eqnarray*}
x_2 + x_3 &\geq & -c \\
x_3 - 2x_2 &\geq & 0\\
x_2 - 2x_3 &\geq & 0
\end{eqnarray*}

These inequalities define a triangular region of the $x_2-x_3$ plane with vertices at $(0,0)$, $(-\frac{2c}{3}, -\frac{c}{3})$, and $(-\frac{c}{3}, -\frac{2c}{3})$. We wish to count the number of lattice points within this rational triangular region. The theory of Ehrhart polynomials tells us this has to be a quasi-polynomial of degree 2. To actually compute it, first note that if $n$ is divisible by $3$, we can use Pick's formula. If $c = 3k$, then the vertices of this triangle are $(-2k, -k)$, $(-k, -2k)$, and $(0,0)$. A simple count shows us that there are $3k$ lattice points on the boundary of this triangle. Moreover, the area of this triangle is equal to $\frac{3k^2}{2}$. Since Pick's formula states that:

$$A = I + \frac{B}{2} - 1$$

where $A$ is the area of the region, $I$ is the number of lattice points in the interior, and $B$ is the number of lattice points on the boundary. Applying this to our triangle, we have that:

$$\dfrac{3k^2}{2} = I + \dfrac{3k}{2} - 1$$

from which we have

$$I = \dfrac{3k^2 -3k +2}{2}$$

The total number of lattice points in the triangle is thus:

$$I + B = \dfrac{3k^2 + 3k + 2}{2}$$

Now, if $c = 3k+1$, let us count how many lattice points are added to the triangular region (i.e., that lie in the triangular region for $n=3k+1$ but don't for $c=3k$). These are simply the lattice points which lie on the line $x_2 + x_3 = -3k-1$ for $-\frac{c}{3} \geq x_2 \geq -\frac{2c}{3}$. Since $x_2$ must be an integer, this implies that $-(k+1) \geq x_2 \geq -2k$. Since for each such $x_2$ we can find a unique integer value of $x_3$, we therefore have that for the case where $c=3k+1$, the total number of lattice points equals:

$$\dfrac{3k^2 + 5k + 2}{2}$$

Likewise, when $c=3k+2$, the number of additional points is simply the number of integral $x_2$ that satisfy $-\frac{c}{3} \geq x_2 \geq -\frac{2c}{3}$, or equivalently, $-(k+1) \geq x_2 \geq -(2k+1)$. Hence, there are $k+1$ additional points, so the total number of lattice points for $c=3k+2$ is:

$$\dfrac{3k^2 + 7k + 4}{2}$$

Each of these corresponds to a unique debt-reachable configuration, so we are done.
\end{proof}

To compute the number of reachable configurations, we need the following result:

\begin{theorem}\label{thm2}
If a chip configuration $(c_1,c_2,c_3)$ of $C_3$ is debt-reachable from a configuration but not reachable from the same configuration, then $(c_1, c_2, c_3)$ is a permutation of $(c,0,0)$ (where $c = c_1 + c_2 + c_3$).
\end{theorem}

\begin{proof}

It is clear that the configurations $(c,0,0)$ (and its cyclic permutations) are not reachable from any other configuration, because there is no possible prior configuration (any such prior configuration would either have the form $(c+2,-1,-1)$ or $(c-1, 2, -1)$). Assume that we are trying to reach some configuration $C'$ that is debt-reachable from configuration $C$. Let $x = (x_1, x_2, x_3)$ be the reduced firing vector of the transition from $C$ to $C'$. If $x_1 = x_2 = x_3 = 0$, $C$ and $C'$ are the same configuration, and we are done. Hence assume that at least one of the $x_i$ is greater than $0$, or alternatively that $x_1 + x_2 + x_3 > 0$. We will induct on the value of $x_1 + x_2 + x_3$; that is, we will show that unless $C' = (c,0,0)$ (or a cyclic permutation), we can fire a vertex of $C$ and get to a configuration where the components of the reduced firing vector sum to $x_1 + x_2 + x_3 - 1$.

As long as $c>3$ (and for $c=1,2,3$ we can check this lemma easily by hand), there will be one vertex with at least $2$ chips on it. Without loss of generality, let this vertex be vertex $1$. If $x_1 > 0$, fire vertex $1$, so we arrive at a new configuration $C''$ where we have to fire vertex $1$ $x_1 - 1$ times, vertex $2$ $x_2$ times, and vertex $3$ $x_3$ times, and we have reduced this to a smaller case. Otherwise, we must have $x_1 = 0$. Similarly, if either vertex $2$ or $3$ has at least $2$ chips on it and if either $x_2$ or $x_3$ is greater than $0$, fire the corresponding vertex and reduce the sum $x_1 + x_2 + x_3$ by one, as before. 

Furthermore, if $x_2 = 0$, then only $x_3$ is greater than $0$. We then claim that $c_3$ must have at least $2$ chips on it, because otherwise the final configuration would be $(c_1 + x_3, c_2 + x_3, c_3 - 2x_3)$, which is not valid if $c_3 < 2x_3$. Hence both $c_3\geq 2$ and $x_3 > 0$, so we can fire vertex 3 and decrease $x_1 + x_2 + x_3$ by one. Likewise, by symmetry we cannot have $x_3 = 0$ and $x_2 > 0$, so both $x_2$ and $x_3$ must be greater than $0$, and consequently, both $c_2$ and $c_3$ must be less than $2$.

There are then at most four possibilities we need to consider; $(c_2, c_3) = (0,0)$, $(0,1)$, $(1,0)$, and $(1,1)$, or equivalently, $(c_1,c_2,c_3) = (c,0,0)$, $(c-1,0,1)$, $(c-1,1,0)$, and $(c-2,1,1)$. We can check that the only case where there are values of $x_2$ and $x_3$ larger than $0$ which result in a legal chip-firing configuration is the case $(c-2,1,1)$, where if we fire both vertices 2 and 3, we end up at $(c,0,0)$. This means, however, that $C' = (c,0,0)$, and hence we are done.
\end{proof}

This immediately leads to the following theorem.

\begin{theorem}\label{thm3}
The number of reachable configurations from the inital configuration $(c,0,0)$ on $C_3$ is equal to the number of debt-reachable configurations when $c$ is not divisible by $3$, and is $2$ less than the number of debt-reachable configurations when $c$ is divisible by $3$. Explicitly, the number of reachable configurations equals:

$$\dfrac{3k^2 + 3k - 2}{2} \mbox{ if } c = 3k$$

$$\dfrac{3k^2 + 5k + 2}{2} \mbox{ if } c = 3k +1$$

$$\dfrac{3k^2 + 7k + 4}{2} \mbox{ if } c = 3k + 2$$
\end{theorem}

\begin{proof}
We can check (via Theorem \ref{triv2}) that $(0,c,0)$ (and similarly, $(0,0,c)$) are only debt-reachable from $(c,0,0)$ when $3|c$. Therefore the number of reachable configurations is just the same as the number of debt-reachable configurations when $c$ is not divisible by $3$, and is $2$ less than the number of debt-reachable configurations when $c$ is divisible by $3$.
\end{proof}

\section{Debt-reachable configurations on arbitrary graphs}\label{sec4}

In this section, we derive asymptotics and properties of the number of configurations debt-reachable from the configuration $(c,0,0,\dots,0)$ on an arbitrary graph $G$.

\begin{theorem}\label{thm4}
The number of configurations debt-reachable from $(c,0,0,\dots,0)$ on an arbitrary graph $G$ is a quasipolynomial in $c$ of degree $n-1$ with period dividing $\kappa(G)$.
\end{theorem}
\begin{proof}
Recall that (Theorem \ref{triv2}) a configuration $C' = (c_1, c_2, \dots, c_n)$ is debt-reachable from $C = (c,0,0,..,0)$ iff $C'-C = L\textbf{x}$ for some vector $\textbf{x} = (x_1, x_2, \dots, x_n)$ in $\mathbb{Z}^n$. Since $(1,1,\dots,1)$ is the only zero eigenvector of $L$, we can impose the additional condition $x_n = 0$. There is then a bijective correspondence between vectors $\textbf{x} \in \mathbb{Z}^n$ with $x_n = 0$ where $C + L\textbf{x}$ has nonnegative components and debt-reachable configurations $C'$. We therefore want to count the number of such $\textbf{x}$ such that $C+L\textbf{x}$ has nonnegative components. Such $x$ are determined by $n$ linear inequalities in the $n-1$ variables $x_1$ through $x_{n-1}$. Each of these inequalities determine half-spaces of $\mathbb{R}^{n-1}$. Of these half-spaces, $n-1$ pass through the origin. The remaining half-space is of the form $\textbf{a}\cdot \textbf{x} \geq -c$ for some vector $a \in \mathbb{Z}^n$. These $n$ half-spaces thus determine a rational $(n-1)$-simplex in $(n-1)$-dimensional space with one vertex at the origin and $n-1$ other vertices at points $c\textbf{v}_i$ for some $\textbf{v}_i \in \mathbb{Q}^n$ independent of $c$, for $1 \leq i \leq n-1$. 

Therefore, if we let $P$ be the simplex formed when $c=1$, the simplex formed for arbitrary $c$ is just the $c$-fold dilate $cP$ of $P$. Since the number of lattice points contained in $cP$ is given by the Ehrhart quasipolynomial of $P$, the number of configurations debt-reachable from $(c,0,0,\dots,0)$ is a quasipolynomial in $c$. Moreover, since $P$ is $(n-1)$-dimensional, this quasipolynomial has degree $n-1$.

It remains to show that the period of this quasipolynomial divides $\kappa(G)$. To do this, it suffices to show that $\kappa(G)\textbf{v}_i \in \mathbb{Z}^{n-1}$ for each $\textbf{v}_i$. Let $L'_i$ denote the reduced Laplacian matrix formed by removing the last column and the $i$th row from $L$. From the construction of $\textbf{v}_i$, we know that $\textbf{v}_i$ satisfies:

$$L'_i\textbf{v}_i = (1,0,\dots,0)^t$$

so

$$\textbf{v}_i = (L'_{i})^{-1}(1,0,\dots,0)^t$$

Since $L'_i$ is an integer matrix, an elementary theorem of matrix algebra says that $(\det L'_i)(L'_i)^{-1}$ is an integer matrix (namely, the cofactor matrix). Therefore $(\det L'_i)\textbf{v}_i$ is an integer vector. By Theorem \ref{mtthm}, $|\det L'_i| = \kappa(G)$, we have that $\kappa(G)\textbf{v}_i$ is an integer vector, as desired.
\end{proof}

Before we continue further, note that debt-reachability is an equivalence relation on the set of chip-firing configurations on a graph $G$ with a total of $c$ chips, since if configuration $C$ is debt-reachable from configuration $C'$, $C'$ is also debt-reachable from configuration $C$ (this is an immediate consequence of Theorem \ref{triv2}). This equivalence relation therefore partitions this set into several subsets. We will call these subsets \textit{blocks}, and let $b_c$ count the number of such blocks for the set of configurations with a total of $c$ chips.

We will now establish some properties about the sequence $b_c$.

\begin{theorem}\label{thm5}
For all $c\geq 1$, $b_c \geq b_{c-1}$.
\end{theorem}
\begin{proof}
Note that if configurations $C = (c_1, c_2, \dots, c_n)$ and $C' = (c'_1, c'_2, \dots, c'_n)$ are not debt-reachable from each other (i.e. lie in different blocks), then the configurations $D = (c_1+1, c_2, \dots, c_n)$ and $D' = (c'_1 + 1, c'_2, \dots, c'_n)$ are likewise not debt-reachable from each other. Also, if $C$ and $C'$ are configurations with $c-1$ chips, then $D$ and $D'$ are configurations with $c$ chips. It therefore follows that $b_c \geq b_{c-1}$, as desired.
\end{proof}

\begin{theorem}\label{thm6}
For all $c$, $b_c \leq \kappa(G)$, and for some sufficiently large $C$, $b_C = \kappa(G)$.
\end{theorem}
\begin{proof}
Consider the integer lattice in $n$-dimensional space containing the points $Lx$ for all vectors $x$ in $\mathbb{Z}^n$ (i.e. it is generated by the rows of $L$). We refer to this lattice as the chip-firing lattice,  since an element $(r_1, r_2, \dots, r_n)$ lies in this lattice iff there is some sequence of chip-firing moves which sends a configuration $(c_1,c_2, \dots ,c_n)$ to $(c_1+r_1, c_2+r_2,\dots, c_n+r_n)$. We denote this lattice by $CFL$.

Note that for any element $\textbf{r} = (r_1, r_2, \dots, r_n)$ in this lattice, $r_1 + r_2 + \dots + r_n = 0$. We can therefore project this lattice onto the subspace where $r_n = 0$ and uniquely identify any element in the original $CFL$ by the first $n-1$ components, $(r_1, r_2, \dots , r_{n-1})$. We denote this modified lattice by $CFL'$. Furthermore, note that since the last row of $L$ is an integer linear combination of the first $n-1$ rows of $L$ (since the sum of all $n$ rows is zero), the lattice $CFL'$ consists of all points $L'\textbf{x}$, where $\textbf{x}$ lies in $\mathbb{Z}^{n-1}$ and $L'$ is the reduced Laplacian matrix.

We then have that the number of elements of $\mathbb{Z}^{n-1}/CFL' = \mathbb{Z}^{n-1}/L'\textbf{x}$ is equal to $\det L'$ (this is a standard fact in the theory of lattices). In other words, there are at most $\det L'$ distinct ``residue classes'' of chip-firing configurations; each possible chip-firing configuration is debt-reachable from one of a particularly chosen set of $\det L'$ configurations. By Theorem \ref{mtthm}, $\det L' = \kappa(G)$, so $b_c \leq \kappa(G)$ for all $c$

Now, $b_c$ is actually counting the number of distinct residue classes of the lattice points in $\mathbb{Z}^{n-1}$ contained in the rational $(n-1)$-simplex defined by the $n$ inequalities $r_1 \geq -c$, $r_i \geq 0$ for $2 \leq i \leq n-1$, and $(r_1+r_2+\dots +r_{n-1}) \leq 0$.

This region scales with $c$, so eventually it will contain a sphere large enough that must contain a fundamental region of $CFL'$. At this point, it will contain a point for every single residue class, so there is some $C$ such that $b_C = \kappa(G)$, as desired.
\end{proof}

\begin{remark}
The finite abelian group $\mathbb{Z}^{n-1}/CFL'$ of equivalence classes modulo the lattice $CFL'$ is commonly known as the \textit{sandpile group} of the graph $G$. Much research into chip-firing on graphs has focused on the properties of these groups; see \cite{Bi, Le} for details.
\end{remark}

As a corollary of Theorems \ref{thm5} and \ref{thm6} we have:

\begin{corollary}\label{thm6_5}
There exists a $C$ such that $b_c = \kappa(G)$ for all $c \geq C$.
\end{corollary}

As another consequence of Theorems \ref{thm5} and \ref{thm6}, we have the following interesting (albeit slightly tangential) result:

\begin{corollary}\label{trees}
If $G$ is a tree, then all chip-firing configurations on $G$ with $c$ chips are reachable from eachother.
\end{corollary}

Finally, we cite the following well-known lemma regarding the approximation of volumes.

\begin{lemma}\label{vollemma}
Let $P$ is a polytope in $\mathbb{R}^n$ with content $V(P)$, and let $L$ be a lattice in $\mathbb{R}^n$ whose fundamental domain has content $\mu(L)$. Then as $c$ tends to infinity, the number of points of $L$ inside $cP$ is asympotically equal to 

$$\dfrac{c^nV(P)}{\mu(L)}$$
\end{lemma}
\begin{proof}
Refer to \cite{Be}.
\end{proof}

We are now ready to derive an asymptotic estimate for the sizes of the blocks in $b_c$.

\begin{theorem}\label{thm7}
The size of each of the blocks in $b_c$ is asymptotically equivalent to

$$\dfrac{1}{\kappa(G)}\binom{c+n-1}{n-1}$$

as $c$ approaches infinity. This in turn is asymptotically equivalent to

$$\dfrac{c^{n-1}}{(n-1)!\kappa(G)}$$
\end{theorem}
\begin{proof}
Note that there are a total of $\binom{c+n-1}{n-1}$ chip-firing configurations with a total of $c$ chips on a graph with $n$ vertices. Since there are a total of $\kappa(G)$ blocks for sufficiently large $c$ (by Corollary \ref{thm6_5}), it suffices to show that all blocks of $b_c$ have asymptotically the same size. 

Let $P_c$ be the polytope determined by the convex hull of the region defined in the proof of Theorem \ref{thm6}. Let $B$ be a specific block of $b_c$, and let $\textbf{v}$ be a chip-firing configuration belonging to $B$. Then if we consider the (uncentered) lattice containing the elements $\textbf{v} + L'\textbf{x}$ for $\textbf{x}\in\mathbb{Z}^{n-1}$, $|B|$ is simply equal to the number of points belonging to this lattice that also lie within $P_c$. 

Now, as mentioned in the proof of \ref{thm6}, the region $P_c$ scales with $c$ (so $P_c = cP_1$). Hence, by Lemma \ref{vollemma}, we have that $|B|$ asymptotically approaches:

$$\dfrac{c^{n-1}V(P_1)}{\mu(\textbf{v}+L'\textbf{x})}$$

Since the volume of the fundamental region of a lattice is unchanged by translating the lattice, $\mu(\textbf{v}+L'\textbf{x})$ is independent of $\textbf{v}$ (in fact, it is equal to $\det L' = \kappa(G)$). Hence $|B|$ asymptotically approaches the same value regardless of the choice of block $B$, and therefore all blocks of $b_c$ have asymptotically the same size, as desired.
\end{proof}

But now, since (by definition) the configurations debt-reachable from a chip-firing configuration are those that belong to the same block as the configuration, we have the following main theorem as an immediate consequence of the preceding result.

\begin{theorem}\label{thm7_5}
The number of configurations debt-reachable from an arbitrary configuration with $c$ chips asymptotically approaches 

$$\dfrac{c^{n-1}}{(n-1)!\kappa(G)}$$

as $c$ approaches infinity.
\end{theorem}

In particular, since by Theorem \ref{thm4} the number of debt-reachable configurations from $(n,0,\dots,0)$ is given by a quasipolynomial of degree $n-1$, this implies that the leading term of this quasipolynomial is always equal to $\dfrac{1}{(n-1)!\kappa(G)}$.

\section{Debt-reachable configurations on $C_n$}\label{sec5}

In this section we apply the results of the previous section to obtain an exact formula for the number of debt-reachable configurations on $C_n$ for an infinite number of configurations of the form $(c,0,\dots,0)$. 

First, note that as a corollary of Theorem \ref{thm7_5} (along with the fact that $\kappa(C_n) = n$), we know that the number of debt-reachable configurations from $(c,0,\dots,0)$ on $C_n$ is asymptotically equal to $\frac{c^{n-1}}{n!}$. This matches up with our explicit formulae for $C_3$, where the number of debt-reachable configurations is asymptotically equivalent to $\frac{c^2}{6}$.

To strengthen this result, we must make some additional observations regarding the numbers $b_c$ in the case of $C_n$. Let $e_i$ represent the configuration where there is one chip on vertex $i$ and zero chips elsewhere, and let $ce_i$ similarly represent the configuration where there are $c$ chips on vertex $i$ and zero chips elsewhere.

\begin{theorem}\label{thm8}
For the case of $G = C_n$, $b_1 = n$.
\end{theorem}
\begin{proof}
We must show that the configuration $e_i$ is not debt-reachable from $e_j$ for all $i\neq j$. Due to the symmetry of $C_n$, it suffices to just show that $e_i$ is not debt-reachable from $e_1$ for $i\neq 1$. We proceed to do this.

Assume $e_i$ is debt-reachable from $e_1$, and let $\textbf{x}$ be a corresponding firing vector. We thus have that:

$$L\textbf{x} = e_i - e_1$$

Writing out $\textbf{x} = (x_1, x_2, \dots, x_n)$, this is equivalent to the following set of $n$ equalities:

\begin{eqnarray*}
x_{n} -2x_1 + x_2 &=& -1 \\
x_{i-1} - 2x_i + x_{i+1} &=& 1\\
x_{j-1} - 2x_j + x_{j+1} &=& 0 \mbox{ for } j \neq 1, i
\end{eqnarray*}

Define $y_i = x_{i+1} - x_i$, where indices are taken modulo $n$ (so $x_{n+1} = x_1$). This system is then equivalent to:

\begin{eqnarray*}
y_1 - y_n &=& -1 \\
y_i - y_{i-1} &=& 1\\
y_{j} - y_{j-1} &=& 0 \mbox{ for } j \neq 1, i
\end{eqnarray*}

Note that we also have the condition $y_1 + y_2 + \dots + y_n = 0$. This system is easy to solve (seeing as the $y_j$ take on only two different values), and has solution $y_j = \frac{i-1}{n}-1$ for $1\leq j \leq i-1$ and $y_j = \frac{i-1}{n}$ otherwise. Since there is no value of $i$ between $2$ and $n$ inclusive for which these $y_i$ are integers, as a consequence, the $x_i$ cannot possibly all be integral, and therefore there is no such firing vector, as desired.

\end{proof}

\begin{corollary}\label{thm8_5}
For $G = C_n$, $b_c = n$ for all $c$.
\end{corollary}
\begin{proof}
Follows from the fact that $\kappa(C_n) = n$ along with Theorems \ref{thm8} and \ref{thm6}.
\end{proof}

\begin{theorem}\label{thm9}
For $C_n$, there exist no pair of distinct indices $i,j$ such that $ce_i$ is not debt reachable from $ce_j$ iff $c$ is relatively prime to $n$.
\end{theorem}
\begin{proof}
We proceed identically to the proof Theorem \ref{thm8}. Such a pair of indices can only exist if there is an integer solution to:

\begin{eqnarray*}
y_1 - y_n &=& -c \\
y_i - y_{i-1} &=& c\\
y_{j} - y_{j-1} &=& 0 \mbox{ for } j \neq 1, i
\end{eqnarray*}

with $y_1 + y_2 + \dots + y_n = 0$. Similarly to before, the solution to this system of equations is  $y_j = c(\frac{i-1}{n}-1)$ for $1\leq j \leq i-1$ and $y_j = c(\frac{i-1}{n})$ otherwise. If $c$ is relatively prime to $n$, there can exist no integer solution for $y_j$ by the same logic as in Theorem \ref{thm8}. On the other hand, if $\gcd(c,n) = r > 1$, then if we set $i = 1+\frac{n}{r}$, we have an integral solution for the $y_j$, as desired.
\end{proof}

We can now prove the main theorem of this section.

\begin{theorem}\label{thm10}
If $c$ is relatively prime to $n$, then the number of configurations debt-reachable from the configuration $(c,0,0,\dots,0)$ on $C_n$ is equal to:

$$\dfrac{1}{n}\binom{c+n-1}{n-1}$$

\end{theorem}
\begin{proof}
Clearly, the number of configurations debt-reachable from the configuration $ce_1$ must equal the number of configurations debt-reachable from the configuration $ce_i$, by the symmetry of $C_n$. Theorem \ref{thm9} shows that if $c$ is relatively prime to $n$, all the configurations $ce_i$ lie in distinct blocks, and theorem \ref{thm8} show that there are always exactly $n$ blocks. Hence, since there are a total of $\binom{c+n-1}{n-1}$ configurations, each block must contain $\frac{1}{n}\binom{c+n-1}{n-1}$ configurations, from which the result follows.

\end{proof}

We can use Theorem \ref{thm10} to deduce the following information about the number of chip-firing configurations for values of $c$ not relatively prime to $n$.

\begin{theorem}\label{thm10_5}
The number of configurations debt-reachable from $(c,0,0,...,0)$ on $C_n$ (for $c$ not necessarily relatively prime to $n$) is equal to

$$\dfrac{1}{n}\binom{c+n-1}{n-1} + f(c \bmod n)$$

where $f(c \bmod n)$ is a constant which depends only on the remainder of $c$ modulo $n$.
\end{theorem}

\begin{proof} By Theorem \ref{thm4}, we know that this number of debt-reachable configurations is a quasipolynomial in $c$ whose period we can take to be $\kappa(C_n) = n$. We also know that the number of debt-reachable configurations must be weakly increasing with $c$ (adding a chip cannot possibly decrease the number of debt-reachable configurations). These two observations along with Theorem \ref{thm10} lead to the theorem statement.
\end{proof}

\section{Reachable configurations on arbitrary graphs}\label{sec6}

In this section, we show that the same asymptotic estimate in Theorem \ref{thm7} for the number of debt-reachable configurations also holds for the number of reachable configurations. Specifically, we have the following theorem.

\begin{theorem}\label{thm12}
The number of configurations reachable from an arbitrary configuration with $c$ chips asymptotically approaches 

$$\dfrac{c^{n-1}}{(n-1)!\kappa(G)}$$

as $c$ approaches infinity.
\end{theorem}

We first prove the following lemma.

\begin{lemma}\label{thm11}
If the configuration $(c_1, c_2, ..., c_n)$ is debt-reachable from the configuration $(c'_1, c'_2, ..., c'_n)$ then the configuration $(c_1 + d_1, c_2 + d_2, ..., c_n + d_n)$ is reachable from configuration $(c'_1 + d_1, c'_2 + d_2, ..., c'_n + d_n)$. (Recall that $d_i$ is the degree of vertex $i$).
\end{lemma}

\begin{proof}
Alternatively, the theorem statement is equivalent to the statement that if two configurations $C$ and $C'$ are debt-reachable from each other, then they are debt-reachable while going at most $d_i$ chips into debt on vertex $i$ at any point. We shall prove this version of the statement.

Let $\textbf{x} = (x_1, x_2, \dots, x_n)$ be the reduced firing vector for the transition from $C$ to $C'$. Consider the following algorithm to reach $C'$ from $C$ without going more than $d_i$ chips in debt on vertex $i$ at any point for all vertices $i$: If there exists an index $i$ such that $x_i > 0$ and such that vertex $i$ is not currently in debt, then fire vertex $i$ and decrease $x_i$ by $1$.

We claim that this algorithm terminates only when we reach the configuration $C'$ (i.e., when all the $x_i = 0$). To see why this is true, assume to the contrary that vertex $i$ is in debt for all $i$ where $x_i > 0$. Let $S$ the set of indices $i$ where $x_i > 0$. We claim now if that we fire all the vertices in $S$ the appropriate number of times (as given by the $x_i$), at least one of these vertices will remain in debt. To prove this, note that if we look at the total number of chips on all the vertices of $S$ (which is currently a negative number), this number can only decrease or remain the same when a vertex in $S$ is fired. Hence, after firing all the vertices in $S$ the appropriate number of times, the total number of chips on vertices of $S$ will be negative, and hence there is at least one vertex in $S$ which is still in debt. But this sequence of chip-firing moves was supposed to reach the (legal) configuration $C'$, which contains no vertex in debt. Therefore we have reached a contradiction, and the statement of the theorem holds, as desired.
\end{proof}

We can now prove Theorem \ref{thm12}.

\begin{proof}[\textbf{Proof of Theorem~\ref{thm12}}]
First, note that for sufficiently large $c$, for any configuration with $c$ chips, it must be possible to reach some configuration $C$ with at least $d_i$ vertices on each vertex $i$. Assume $C$ has the form $(c_1, c_2, \dots, c_n)$. By Lemma \ref{thm11}, the number of configurations reachable from $C$ is greater than or equal to the number of configurations debt-reachable from $C' = (c_1 - d_1, c_2 - d_2, \dots, c_n - d_n)$. This configuration contains $c - 2E$ chips, where $E$ is the number of edges of $G$ (since $\sum d_i = 2E$). By Theorem \ref{thm7_5}, the number of configurations debt-reachable from $C'$ is asympotically equal to $(c-2E)^{n-1}/((n-1)!\kappa(G))$ which in turn is asymptotically equal to $c^{n-1}/((n-1)!\kappa(G))$. Hence, the number of configurations reachable from this configuration with $c$ chips is asympotically greater than or equal to $c^{n-1}/((n-1)!\kappa(G))$. However, the  number of reachable configurations must also be less than or equal to the number of debt-reachable configurations, which is also asymptotically equal to $c^{n-1}/((n-1)!\kappa(G))$ by Theorem \ref{thm7_5}. Therefore the number of configurations reachable from this configuration is asymptotically equal to 

$$\dfrac{c^{n-1}}{(n-1)!\kappa(G)}$$

as $c$ approaches infinity, as desired.
\end{proof}

\section{Towards quasipolynomiality of the number of reachable configurations}\label{sec7}

In Theorem \ref{thm4} we show that the number of configurations debt-reachable from the configuration $(c,0,\dots,0)$ is a quasipolynomial in $c$. It is natural to then ask the same question about reachable configurations; does the number of reachable configurations exhibit any quasipolynomial behavior? For the case of $C_3$, Theorem \ref{thm3} answers this in the affirmative. What about for other graphs $G$?

For $C_4$, computer-aided computation of the number of reachable configurations from the configuration $(c,0,\dots,0)$ for $1 \leq c \leq 100$ suggests that this number is a quasipolynomial for $c\geq 4$. Similar computations for other graphs ($C_5$, $C_6$, $P_2$, $P_3$, $K_4$) suggests the following conjecture.

\begin{conjecture}\label{conj1}
The number of configurations reachable from $(c,0,0,\dots,0)$ on an arbitrary graph $G$ is a quasipolynomial in $c$ of degree $n-1$ for sufficiently large $c$.
\end{conjecture}

While we are unable to prove this conjecture, we prove the following characterization theorem for configurations which are debt-reachable but not reachable

\begin{theorem}\label{thm13}
Let $C'$ be a configuration which is debt-reachable but not reachable from configuration $C$. Let $x = (x_1, x_2, \dots, x_n)$ be the reduced firing vector between $C$ and $C'$. Now, let $D$ be a reachable configuration from $C$ with reduced firing vector $y = (y_1, y_2, \dots, y_n)$. Then there exists a unique choice of $D$ with the following two properties
\begin{enumerate}[i.]
\item
For all $i$, $y_i \leq x_i$
\item
If $(x_i - y_i) > 0$, then vertex $i$ cannot fire in configuration $D$.
\end{enumerate}
\end{theorem}
\begin{proof}
Clearly there exists at least one such configuration $D$, since we can just keep on firing vertices $v$ with $x_v > 0$ until there are no longer any legal firing moves remaining. Thus, we need only show uniqueness. Assume there are two distinct such configurations, $D_1$ and $D_2$. Consider the vertices in $D_1$ that were fired more times than the vertices in $D_2$ (without loss of generality these exist; otherwise, swap $D_1$ and $D_2$). We claim that in configuration $D_2$, we can fire at least one of these vertices. 

To prove this, consider the first time in the transition from $C$ to $D_1$ when a vertex is fired for the $(a+1)$-th time when this vertex is only fired a total of $a$ times when reaching $D_2$. We claim that this vertex is fireable in configuration $D_2$. This follows since the only operation that can reduce the number of chips on this vertex is firing this vertex, so therefore, the number of chips on this vertex has to be at least as large as the number of chips on this vertex when it was fired for the $(a+1)$-th time when reaching $D_1$ (since all prior moves when reaching $D_1$ have also been made at some point in reaching $D_2$). Then, since this vertex was fireable in configuration $D_1$, this vertex must be fireable in configuration $D_2$, as desired. 
\end{proof}

In essence, this theorem states that every configuration which is debt-reachable but not reachable corresponds uniquely to a ``most fired'' reachable configuration and a firing vector entirely supported on the non-fireable vertices of this configuration. We hope that this characterization theorem could be used to prove that the number of configurations that are debt-reachable but not reachable is a quasipolynomial in $c$; together with Theorem \ref{thm4}, this would imply Conjecture \ref{conj1}. Note the similarities between this line of reasoning and the line of reasoning used in Theorems \ref{thm2} and \ref{thm3}.

In addition to this characterization theorem, we also prove two lemmas exploring the relationship between debt-reachability and reachability. The first lemma shows that it does not matter what order we fire vertices in when trying to reach a reachable configuration.

\begin{lemma}\label{lem1}
Assume configuration $C'$ is reachable from configuration $C = (c_1, c_2, \dots, c_n)$, and this transition has reduced firing vector $(x_1, x_2, \dots, x_n)$. Then, if $x_i > 0$ and $c_i \geq d_i$, let $C''$ be the configuration obtained by firing vertex $i$. Then configuration $C'$ is reachable from $C''$.
\end{lemma}

\begin{proof}
Since configuration $C'$ is reachable from configuration $C$, there is a sequence of legal chip-firing moves to get from $C$ to $C'$. Let these moves (in order) be firing the vertices $v_1, v_2, \dots, v_r$. Let $v_j$ be the first occurrence of vertex $i$ in this sequence (we know that there must be an occurrence of vertex $i$ in this sequence because $x_i > 0$ in the reduced firing vector). We claim that the sequence $v_j, v_1, v_2, \dots, v_{j-1}, v_{j+1}, \dots, v_r$ is also a sequence of legal chip-firing moves (and since it is a permuation of the previous sequence, it also sends $C$ to $C'$). Note that since firing vertex $v_j = i$ can only increase the number of chips on other vertices, and since none of vertices $v_1$ through $v_{j-1}$ are equal to vertex $i$, the sequence of chip-firing moves $v_j, v_1, v_2, \dots, v_{j-1}$ must be a sequence of legal chip-firing moves (the first move is legal by the conditions of the theorem statement). Then, because $v_j, v_1, v_2, \dots, v_{j-1}$ is a permutation of $v_1, v_2, \dots, v_j$, after $j$ operations the configuration is the same as it was originally after $j$ operations. Since the remaining moves $v_{j+1}, \dots, v_r$ have not changed, they therefore still remain legal.
\end{proof}

The second lemma (along with its corollary) shows that we never need to fire every vertex to reach a configuration. As a consequence, if we can reach a configuration, we can reach it by firing each vertex the number of times given in the reduced firing vector.

\begin{lemma}\label{lem2}
If configuration $C'$ is reachable from configuration $C = (c_1, c_2, \dots, c_n)$ through a sequence of moves that has firing vector $(x_1, x_2, \dots, x_n)$ where all $x_i > 0$, then it is also reachable from configuration $C$ through a sequence of moves that has firing vector $(x_1-1, x_2-1, \dots, x_n-1)$.
\end{lemma}

\begin{proof}
Again, let the sequence of chip-firing moves to get from $C$ to $C'$ be firing the vertices $v_1, v_2, ..., v_r$, in that order. Since $x_i>0$ for all $i$, we know that each of the vertices $1$ through $n$ appears at least once in this sequence. Without loss of generality, let the last occurence of vertex $i$ occur after the last occurrence of vertex $j$ iff $i<j$ (we can attain this by simply relabeling the vertices). Now, let $v_{k_n}$ be the last occurence of the vertex $n$, and let $v_{k_i}$ be the first occurence of the vertex $i$ after $v_{k_n}$ for $1\leq i \leq n-1$. For instance, if our sequence was

$$1, 4, 4, 2, 2, 3, 1, 2, 1, 1$$

then $k_4 = 3$, $k_1 = 7$, $k_2 = 4$, and $k_3 = 6$. We claim that the sequence obtained by removing $v_{k_i}$ for all $1\leq i \leq n$ is still a legal sequence, proving the lemma. 

To show this, assume to the contrary that after we perform this operation, vertex $v_p = i$ is in debt. Let the set $S$ consist of the vertices that occur before vertex $v_p$ in the original firing sequence that were removed. By the above construction, $S$ must include vertex $i$, since vertex $v_{k_i}$ must occur before vertex $i$. Also, since we remove each vertex just once from the sequence, $S$ contains no repeats. 

Now firing vertex $j$ only decreases the number of vertices on vertex $i$ if $j=i$. Since firing all $n$ vertices does not change the number of chips on vertex $i$, firing a subset of the $n$ vertices which contains vertex $i$ cannot increase the number of chips on vertex $i$. Since $S$ is such a set, firing all the vertices in $S$ cannot increase the number of chips on vertex $v_p$. But this means that vertex $v_p$ must have been in debt in the original sequence as well, which is impossible. Therefore our new sequence is a legal chip-firing sequence, as desired.
\end{proof}

\begin{corollary}
Assume configuration $C'$ is reachable from configuration $C = (c_1, c_2, \dots, c_n)$, and let $x = (x_1, x_2, \dots, x_n)$ be the reduced firing vector for the corresponding transition. Then we can reach configuration $C'$ from configuration $C$ by firing vertex $i$ exactly $x_i$ times for each $1 \leq i \leq n$.
\end{corollary}

\section{Remarks and open problems}
The most relevant open problem is Conjecture \ref{conj1}. The author of this paper believes that there should be a proof of this conjecture using Theorem \ref{thm13} and the two Lemmas \ref{lem1} and \ref{lem2}. Efforts along these lines have not been successful thus far, however.

If Conjecture \ref{conj1} is answered in the affirmative, there are a host of related questions that can be asked. For instance, how large must $c$ be for the number of reachable configurations to exhibit quasipolynomial behavior? Preliminary investigations suggest that $c\geq 2E-n$ should suffice, but it is unclear if a better bound can be found. Similarly, what properties does the period of the resulting quasipolynomial satisfy? Note that for $C_5$ it appears that the period of the associated quasipolynomial is $10$, so the period likely does not have to divide $\kappa(G)$ as in Theorem \ref{thm4}. We can also ask what these quasipolynomials look like in the case of cycle graphs $C_n$, extending the result of Theorem \ref{thm10}.

Even in analyzing the number of debt-reachable configurations has met with more success, there are several questions that remain to be answered. For instance, what are the exact periods in Theorem \ref{thm4} (we know only that they divide $\kappa(G)$)? For what other graphs (aside from $C_n$) can we find particularly simple forms for the number of debt-reachable configurations? Is there a particularly simple combinatorial interpretation of the numbers $b_i$ for small values of $i$ (for instance, $b_1$)?

Finally, it would be interesting to extend these methods to the case of particular classes of infinite graphs, for instance an infinite path or an infinite square lattice. Such research could perhaps be relevant to various chip-firing and sandpile model problems in statistical physics. 

\section*{Acknowledgements}

This work was performed in part during MIT's SPUR (Summer Program for Undergraduate Research) in the summer of 2010. The author would also like to thank Dr. Gregg Musiker and Hoda Bidkhori for introducing him to this problem.

\end{document}